\documentclass{article}
\usepackage{amsmath,amssymb}
\usepackage{amsthm}
\usepackage{txfonts}
\newtheorem{Def}{Definition}[section]

\newtheorem{Thm}[Def]{Theorem}
\newtheorem{Rem}[Def]{Remark}
\newtheorem{Prop}[Def]{Proposition}

\newtheorem{Cor}[Def]{Corollary}

\newcommand{\Sl}{\mathfrak{sl}_2}
\newcommand{\Hei}{\mathfrak{H}}
\newcommand{\LA}{\mathfrak{L}}
\newcommand{\LC}{\mathfrak{C}}
\newcommand{\Lg}{\mathfrak{g}}
\newcommand{\La}{\mathfrak{a}}

\newcommand{\Ak}{(\mathrm{A}_k)}
\newcommand{\Bk}{(\mathrm{B}_k)}
\newcommand{\Ck}{(\mathrm{C}_k)}
\newcommand{\Pk}{(\mathrm{P}_k)}
\newcommand{\Qk}{(\mathrm{Q}_k)}
\newcommand{\V}{\fontfamily{pnc}\selectfont \textit{v}}
\newcommand{\W}{\fontfamily{pnc}\selectfont \textit{w}}

\usepackage{booktabs}

\begin{document}

\title{Bernoulli-type Relations \\ in Some Noncommutative Polynomial Ring}

\author{\textbf{Shunsuke MURATA}}

\date{Institute of Mathematics, University of Tsukuba \\
Tsukuba, Ibaraki, 305-8571, Japan}
\maketitle

\thispagestyle{empty}

\begin{abstract} 
We find particular relations which we call "Bernoulli-type" in some noncommutative polynomial ring with a single nontrivial relation. More precisely, our ring is isomorphic to the universal enveloping algebra of a two-dimensional non-abelian Lie algebra. From these Bernoulli-type relations in our ring, we can obtain a representation on a certain left ideal with the Bernoulli numbers as structure constants.
\end{abstract}

\vspace{1\baselineskip}

\begin{center}
\end{center}
\vspace{1\baselineskip}

\section{Introduction}

The Bernoulli numbers are a sequence of rational numbers with connections to many branches of mathemathics. 
Especially, they are closely related to the values of the Riemann zeta function at negative integers \cite{AK}, \cite{AT}. 
In this paper, we show a certain connection between some noncommutative polynomial ring and the Bernoulli numbers. 
We let $K[x,y]$ be a noncommutative polynomial ring in two indeterminates $x,y$ over a field $K$ of characteristic zero. 
Now, we define $I= \langle xy-yx-x \rangle$ to be the ideal of $K[x,y]$ generated by $xy-yx-x$, and let $A$ be $K[x,y]/I$, the quotient of $K[x,y]$ by $I$.
Again we use $x$, $y$ as $\bar{x} = x+I$, $\bar{y} = y +I$ respectively (if there is no confusion). 
We note that $A$ is isomorphic to the universal enveloping algebra of a two-dimensional non-abelian Lie algebra (cf. Remark \ref{cong}). 
Then, our main result is the following: 
\vspace{1\baselineskip}

\textsc{Theorem} (Bernoulli-type relations)

Let $A$ be as above.
We put 
\begin{equation*}
\W_{k,\ell} = (xy^{k}-y^{k}x)x^{\ell} \in A \quad(k \geq 1, \ell \geq 0).
\end{equation*}
Then, the following relations hold.
\begin{align*}
x\W_{k, \ell} &= \sum_{i=1}^k \binom{\,k\,}{i}\W_{k, \ell +1} \\
y\W_{k, \ell} &= \frac{k}{k+1}\W_{k+1, \ell} -
                              \sum_{i=1}^k \frac{1}{k+1} \binom{k+1}{i}B_{k+1-i}\W_{i, \ell}
\end{align*}
\hfill$\Box$
\vspace{1\baselineskip}

Here, we note that the above $B_{k+1-i}$ mean the Bernoulli numbers. 
Hence, we call the above relations 
\textquoteleft \textquoteleft Bernoulli-type relations". 
Put $W= \oplus_{m \geq 1, n \geq 0}Kx^{m}y^{n} \subseteq A$, which is a direct sum by PBW theorem. 
Then $W$ becomes a two-sided ideal of $A$. 
Using the Bernoulli-type relations, we can obtain that $W$ is generated by $\{ \W_{k,\ell} \}_{k \geq 1, \ell \geq 0}$.
We can also see that $\{ \W_{k, \ell} \}_{k \geq 1, \ell \geq 0} $ is a basis of $W$.   
\vspace{1\baselineskip}

Here we will explain Bernoulli-type relations in terms of Lie algebras. 
For the explanation, we start to explain our motivation of this study. 
We began this study with \cite{BMY} written about some factorizations in universal enveloping algebras. 
In \cite{BMY}, they deal with universal enveloping algebras of three-dimensional Lie algebras. 
Then they obtained certain general relations.
Let $\LA$ be a three-dimensional Lie algebra over $K$ and denote by $U(\LA)$ the universal  enveloping algebra of $\LA$. 
Assume that $\LA$ is generated by two elements $x,y$. 
Then, the general relations in $U(\LA)$ are given as follows: 
\begin{align*}
   &\Ak \quad yxy^k \equiv \frac{k}{k+1}xy^{k+1} + \frac{1}{k+1}y^{k+1}x  \quad  ( \mathrm{mod} U_k) , \\
   &\Bk \quad y^kxy \equiv \frac{1}{k+1}xy^{k+1} + \frac{k}{k+1}y^{k+1}x  \quad  ( \mathrm{mod} U_k) , \\
   &\Ck \quad yU_k \subseteq U_{k+1} , \; U_ky \subseteq U_{k+1} , \text{where} \\
   &\quad U_k = \displaystyle \sum_{0 \le m \le k} (K xy^m + K y^mx + K y^m ) \quad (k \ge 0).
\end{align*}
\hfill$\Box$

Then, the remainder terms, $u=\sum_{1 \le p,q,r \le k} a_{p} xy^{p}+ b_{q} y^{q}x + c_{r} y^{r}+ dx \in U_k$ with $a_{p}, b_{q}, c_{r}, d  \in K $, of $\Ak$, $\Bk$ are determined according to the generators $x,y$ and types of $\LA$. 
In the paper \cite{BMY}, they determine some exact terms of $u$ along with a classification of $\LA$ in Jacobson's book \cite{Ja}. 

Here we roughly introduce the classification. 
We put $\LA = Ke \oplus Kf \oplus Kg$ with its basis $(e,f,g)$. 
Let $\LA'$ be the derived ideal of $\LA$ and $\LC$ be the center of $\LA$. 
Then the classification is roughly given as follows: 
\vspace{1\baselineskip}

(a) If $\LA'=0$, $\LA$ is abelian. 
\vspace{1\baselineskip}

(b) If dim$\LA' = 1$ and $\LA' \subseteq \LC$, the multiplication table of the basis is
$$[e,f]=g, \; [e,g]=[f,g]=0 .$$ 

(c) If dim$\LA' = 1$ and $\LA' \nsubseteq \LC$, the multiplication table of the basis is
$$[e,f]=e, \; [e,g]=[f,g]=0 .$$ 

(d) If dim$\LA' = 2$, the multiplication tables of the basis are
\begin{align*}
&\text{(d)-($\alpha$)} \quad [e,f]=0, \; [e,g]= e, \; [f,g]= \alpha f, \\
&\text{(d)-($+$) } \quad [e,f]=0, \; [e,g]=e+f, \;  [f,g]=f,
\end{align*}
\quad \; where $\alpha$ in $K^{\times}$. Different choice of $\alpha$ give different algebras unless $\alpha \alpha' =1$ . 
\vspace{1\baselineskip}

(e) dim$\LA' = 3$, the multiplication table of the basis is
$$[e,f]=g, \; [g,e]=2e, \; [g,f]=-2f.$$ 

\vspace{1\baselineskip}

In the type (d) or (e), we suppose that $K$ is algebraically closed (just for our rough explanation). 
As is well-known, the type (b) means a Heisenberg Lie algebra $\Hei_K$ and the type (e) means a special linear Lie algebra $\Sl(K)$. In the paper \cite{BMY}, they determined the exact terms of $u$ for $\Hei_F$ or $\Sl(F)$ with the above generators $e, f$ including the case if $F$ is a field of characteristic zero. 
They also showed that $\LA$ can not be two generated if $\LA$ is the type (a) or the type (d)-($\alpha = 1$). 
Hence, we were interested in determining the terms in $U_k$ for the remaining type of $\LA$. 
For our purpose, we explain some results in the author's master thesis \cite{Mu}. 
Since the paper is written in Japanese, we introduce its summary here.  
In \cite{Mu}, we obtained some properties between $u$ and the types of $\LA$, and determined the exact terms of $u$ if $\LA$ is the type (d)-($+$). 
The properties between $u$ and the types of $\LA$ are given as follows:
\begin{itemize}
\item We always have $u=0$ inspite of generators if $\LA$ is the type (b).

\item We always have $u \neq 0$ inspite of generators if $\LA$ is the type (e).

\item We can get $u=0$ according to some special generators if $\LA$ is the type (c) or (d). 
 (It means that we can also get $u \neq 0$ according to another generators.)
\end{itemize}

The exact terms of $u$ are determined  if $\LA$ is the type (d)-($+$) with the generators $e$ and $g$. 
The formulas in $U(\LA)$ are given as follows: 
\begin{align*}
\Pk \quad &geg^{k} = \frac{k}{k+1}eg^{k+1} + \frac{1}{k+1}g^{k+1}e  
      -eg^k + \frac{1}{k+1} \sum_{i=0}^k  \binom{\,k+1\,}{i}g^ie,  \\
\Qk \quad &g^{k}eg = \frac{1}{k+1}eg^{k+1} + \frac{k}{k+1}g^{k+1}e 
     + \frac{1}{k+1} \sum_{i=0}^k (-1)^{k+1-i}\binom{\,k+1\,}{i}eg^{i} + g^{k}e. 
\end{align*}
\hfill$\Box$
\vspace{1\baselineskip}

These are the almost all results written in \cite{Mu}. 
After we obtained these results, we could establish the formulas if $\LA$ is the type (c) with the generators $e+g$ and $f+g$ in the above classification. 
Then we noticed that our formulas can be reduced to the two-dimensional case. 
That is, we put $L= Kx \oplus Ky$ as a two-dimensional Lie algebra satisfying $[x,y]=x$ and denote by $U(L)$ the universal enveloping algebra of $L$.
Then, the formulas in $U(L)$ are given as follows: 
\vspace{1\baselineskip}

\begin{align*}
\Pk \quad yxy^k &= \frac{k}{k+1} \, xy^{k+1} + \frac{1}{k+1} \, y^{k+1}x \\
  & \quad -  \frac{1}{k+1} \sum_{i=1}^{k} \, \binom{\,k+1 \,}{i} B_{k+1-i} \, xy^{i}
           +  \frac{1}{k+1} \sum_{i=1}^{k}  \, \binom{\,k+1 \,}{i} B_{k+1-i} \, y^{i}x, \\
\Qk \quad y^kxy &= \frac{1}{k+1} \, xy^{k+1} + \frac{k}{k+1} \, y^{k+1}x \\
 & \quad + \frac{1}{k+1} \sum_{i=1}^{k} (-1)^{k+1-i} \, \binom{\, k+1 \,}{i} B_{k+1-i} \, xy^{i}\\ 
 & \quad - \frac{1}{k+1} \sum_{i=1}^{k} (-1)^{k+1-i} \, \binom{\, k+1 \,}{i} B_{k+1-i} \, y^{i}x.
\end{align*}
\hfill$\Box$
\vspace{1\baselineskip}

At first these formulas were shown without the Bernoulli-type relations. 
But using the Bernoulli-type relations, we can easily show the formulas. 
Therefore, we use the Bernoulli-type relations to show the formulas in this paper. 
\vspace{1\baselineskip}

We will review the Bernoulli numbers $B_n$ with $B_1=1/2$ in Section 2. 
In Section 3, we will show the Bernoulli-type relations and study $W$ introduced before. 
In Section 4, we will show the above formulas and explain a connection to Lie algebras. 
We also mention that $U(L)$ is isomorphic to $A$, and that $A = \oplus_{m \geq 1, n \geq 0} Kx^{m}y{n}$ by PBW theorem. 
\vspace{1\baselineskip}

\section{Preliminaries}

In this paper, $K$ is a field of characteristic zero. 
We denote a left hand side (resp: right hand side) by (LHS) (resp: (RHS)). 
We also denote by $B_n$ the Bernoulli numbers.

In this section, we review the Bernoulli numbers with $B_1=1/2$. 
We aim a self-contained explanation in this paper. 
Thus we confirm our setting here. 

\begin{Def}\label{BN} $(\mathrm{The \, Bernoulli \, numbers})$

We define the Bernoulli numbers $B_n$ recursively as follows:
\begin{equation*}
\sum_{i=0}^{n}\binom{n+1}{i}B_{i} = n+1 .
\end{equation*}
\hfill$\Box$
\end{Def}

\begin{Rem}\
In general, the Bernoulli numbers are also given by a generating function. 
The generating function in our condition is given as follows: 
\begin{equation*}
\frac{te^t}{e^t-1} = \sum_{n=0}^{\infty}B_n\frac{t^n}{n!}.
\end{equation*}
\hfill$\Box$
\end{Rem}

Here we describe the Bernoulli numbers up to $n=10$.  
\begin{center}
Figure.1 \; The Bernoulli numbers
 \begin{tabular}{cccccccccccc} \toprule
   $n$&$0$&$1$&$2$&$3$&$4$&$5$&$6$&$7$&$8$&$9$&$10$ \\ \midrule
   $B_n$&$1$&$\frac{1}{2}$&$\frac{1}{6}$&$0$&$-\frac{1}{30}$&$0$&$\frac{1}{42}$&$0$&$-\frac{1}{30}$&$0$&$\frac{5}{66}$ \\ \bottomrule
 \end{tabular}
\end{center}
\vspace{1\baselineskip}

As well known, there are the other type of the Bernoulli numbers. 
If we denote by $\hat{B}_n$ the Bernoulli numbers with $\hat{B}_1=-1/2$, 
then $\hat{B}_n$ are given by $(-1)^{n}B_n$ for $n \geq 0$. 

\vspace{1\baselineskip}

\begin{Rem}
\; In the first half of eighteenth century, the Bernoulli numbers were discoverd around the same time by Jacob Bernoulli and Kowa Seki independently. At first, both Bernoulli and Seki took $B_1=1/2$. Hence, historically, our definition is an original version.
\hfill$\Box$
\end{Rem}
\vspace{1\baselineskip}

\section{Bernoulli-type relations and the ideal $W$ }

In this section, we show the main theorem and some corollaries. 
Now, we set $A = K[x,y]/I$, where $K[x,y]$ is a noncommutative polynomial ring in two indeterminates $x,y$ and $I = \langle xy-yx-x \rangle$ is the two-sided ideal of $K[x,y]$ generated by $xy-yx-x$. 
At first, we confirm several elementary formulas for proving the main theorem.

\begin{Prop} \label{prop}
$\mathrm{(i)}$ 
For integers $k \geq i \geq j \geq 0$, we have
\begin{equation*}
\binom{\, k \,}{i} \binom{\, i \,}{j} = \binom{\, k \,}{j} \binom{\, k-j \, }{i-j}.
\end{equation*}

$\mathrm{(ii)}$
Let $A$ be as above. Then the following formula holds.
\begin{equation*}
xy^k = \sum_{i=0}^k \binom{\,k\,}{i} y^ix
\end{equation*}

$\mathrm{(iii)}$
Let $A$ be as above. Then the following formula holds.
\begin{equation*}
y^kx = \sum_{i=0}^k (-1)^{k-i} \binom{\,k\,}{i} xy^i
\end{equation*}

\end{Prop}

\begin{proof}
$\mathrm{(i)}$ 
We can calculate
\begin{align*}
\binom{\,k\,}{i} \binom{\, i \,}{j} &= \frac{k!}{(k-i)!i!}\, \frac{i!}{(i-j)!j!} \\
                                                &= \frac{k!}{(k-j)!j!}\, \frac{(k-j)!}{\{(k-j)-(i-j)\}!(i-j)!} \\
                                                &= \binom{\,k\,}{j} \binom{\, k-j \,}{i-j} .
\end{align*}

$\mathrm{(ii)}$ 
Since $xy= yx+x=y(x+1)$, we can calculate
\begin{align*}
xy^k &= (yx+x)y^{k-1} 
          = (y+1)xy^{k-1} \\
        &= \dots \\
        &= (y+1)^{k}x 
          = \sum_{i=0}^{k} \binom{\,k\,}{i} y^{i}x . \\
\end{align*}

$\mathrm{(iii)}$ 
Since $yx=xy-x=x(y-1)$, we can calculate
\begin{align*}
y^kx &= y^{k-1}(xy-x) 
          = y^{k-1}x(y-1) \\
        &= \dots \\
        &= x(y-1)^{k} 
          = \sum_{i=0}^{k}(-1)^{k-i} \binom{\,k\,}{i} xy^{i} . \\
\end{align*}

Therefore, we obtain the desired results.
\end{proof}

Now, we prove the main theorem.

\begin{Thm}\label{thm}

Let $A$ be as above.
We take 
\begin{equation*}
\W_{k,\ell} = (xy^{k}-y^{k}x)x^{\ell} \in A  \quad (k \geq 1, \ell \geq 0).
\end{equation*}
Then, the following relations hold.
\begin{align*}
(BR1) \quad x\W_{k, \ell} &= \sum_{i=1}^k \binom{\,k\,}{i}\W_{k, \ell +1} \\
(BR2) \quad y\W_{k, \ell} &= \frac{k}{k+1}\W_{k+1, \ell} -
                              \frac{1}{k+1} \sum_{i=1}^k \binom{\,k+1\,}{i}B_{k+1-i}\W_{i, \ell}
\end{align*}
\end{Thm}

\begin{proof}
At first, we show (BR1). 
Using Proposition \ref{prop} (ii), we can compute
\begin{align*}
x\W_{k, \ell} &= x(xy^{k} - y^{k}x)x^{\ell} \\
                     &= \{ x(xy^{k}) - (xy^{k})x \}x^{\ell} \\
                     &= \Biggl \{ x \Biggr( \sum_{i=0}^{k} \binom{\,k\,}{i}y^{i}x \Biggl ) 
                                       - \Biggl ( \sum_{i=0}^{k} \binom{\,k\,}{i}y^{i}x \Biggr )x \Biggr \}x^{\ell} \\
                     &= \Biggl \{ \sum_{i=0}^{k} \binom{\,k\,}{i} xy^{i} 
                                       - \sum_{i=0}^{k} \binom{\,k\,}{i} y^{i}x  \Biggr \} x^{\ell+1} \\
                     &= \sum_{i=0}^{k} \binom{\,k\,}{i} ( xy^{i} - y^{i}x) x^{\ell+1} \\
                     &= \sum_{i=0}^{k} \binom{\,k\,}{i} \W_{i, \ell +1}. \\
\end{align*}
Next, we show (BR2) by computing from (RHS) to (LHS). 
Using Proposition \ref{prop} (ii), we can compute
\begin{align*}
(RHS) &=\frac{k}{k+1} \W_{k+1, \ell} 
 - \frac{1}{k+1} \sum_{i=1}^{k} \binom{\,k+1\,}{i}B_{k+1-i}\W_{i, \ell} \\
&= \frac{k}{k+1} (xy^{k+1}-y^{k+1}x)x^{\ell} 
   - \frac{1}{k+1} \sum_{i=1}^{k} \binom{\,k+1\,}{i}B_{k+1-i}  (xy^{i}-y^{i}x)x^{\ell} \\
&= \frac{k}{k+1} \Biggl \{ \sum_{i=0}^{k+1}\binom{\, k+1 \,}{i} y^{i}x -y^{k+1}x \Biggr \}x^{\ell} \\
&\quad - \frac{1}{k+1} \sum_{i=1}^{k} \binom{\,k+1\,}{i}B_{k+1-i} 
                \Biggr \{ \sum_{j=0}^{i} \binom{\, i \,}{j} y^{j}x-y^{i}x \Biggl \} x^{\ell} \\
&= \frac{k}{k+1} \sum_{i=0}^{k}\binom{\, k+1 \,}{i} y^{i}x^{\ell+1} 
   - \frac{1}{k+1} \sum_{i=1}^{k} \binom{\,k+1\,}{i}B_{k+1-i}
              \sum_{j=0}^{i-1} \binom{\, i \,}{j} y^{j}x^{\ell+1}. \\
\intertext{We divide (RHS) into three terms such as $x$ and $y^{k}x$ and otherwise. Then we have}
(RHS)&= \frac{k}{k+1}\binom{\, k+1 \,}{k} y^{k}x^{\ell+1} 
   +\frac{k}{k+1} \sum_{i=1}^{k-1}\binom{\, k+1 \,}{i} y^{i}x^{\ell+1} 
   +\frac{k}{k+1} \binom{\, k+1 \,}{0} y^{0}x^{\ell+1} \\
&\quad - \frac{1}{k+1} \sum_{i=2}^{k} \binom{\,k+1\,}{i}B_{k+1-i}
              \sum_{j=1}^{i-1} \binom{\, i \,}{j} y^{j}x^{\ell+1}
  - \frac{1}{k+1} \sum_{i=1}^{k} \binom{\,k+1\,}{i}B_{k+1-i} \binom{\,i\,}{0} y^{0}x^{\ell+1}. \\
\intertext{Since we can replace $B_{k+1-i}$ with $B_{i}$ in the last term, we have }
(RHS)&= ky^{k}x^{\ell+1}  
   + \frac{k}{k+1} \sum_{i=1}^{k-1}\binom{\, k+1 \,}{i} y^{i}x^{\ell+1} 
   + \frac{k}{k+1}x^{\ell+1} \\
&\quad - \frac{1}{k+1} \sum_{i=2}^{k} \binom{\,k+1\,}{i}B_{k+1-i}
              \sum_{j=1}^{i-1} \binom{\, i \,}{j} y^{j}x^{\ell+1} 
   - \frac{1}{k+1} \sum_{i=1}^{k} \binom{\,k+1\,}{i}B_{i} x^{\ell+1} \\
&=  ky^{k}x^{\ell+1} 
   + \frac{k}{k+1} \sum_{i=1}^{k-1}\binom{\, k+1 \,}{i} y^{i}x^{\ell+1} 
   - \frac{1}{k+1} \sum_{i=2}^{k} \binom{\,k+1\,}{i}B_{k+1-i}
              \sum_{j=1}^{i-1} \binom{\, i \,}{j} y^{j}x^{\ell+1} \\
   &\quad + \frac{k}{k+1}x^{\ell+1}
   - \frac{1}{k+1} \sum_{i=0}^{k} \binom{\,k+1\,}{i}B_{i} x^{\ell+1} 
   + \frac{1}{k+1} \binom{\,k+1\,}{0}x^{\ell+1}.\\
\intertext{In the fifth term, using Definition \ref{BN}, we get}
(RHS)&=  ky^{k}x^{\ell+1} 
   + \frac{k}{k+1} \sum_{i=1}^{k-1}\binom{\, k+1 \,}{i} y^{i}x^{\ell+1} 
   - \frac{1}{k+1} \sum_{i=2}^{k} \binom{\,k+1\,}{i}B_{k+1-i}
              \sum_{j=1}^{i-1} \binom{\, i \,}{j} y^{j}x^{\ell+1} \\
   &\quad + \frac{k}{k+1}x^{\ell+1}
   - \frac{1}{k+1} (k+1) x^{\ell+1} 
   + \frac{1}{k+1} x^{\ell+1}.\\
&=  ky^{k}x^{\ell+1} 
   + \frac{k}{k+1} \sum_{i=1}^{k-1}\binom{\, k+1 \,}{i} y^{i}x^{\ell+1} 
   - \frac{1}{k+1} \sum_{i=2}^{k} \binom{\,k+1\,}{i}B_{k+1-i}
              \sum_{j=1}^{i-1} \binom{\, i \,}{j} y^{j}x^{\ell+1}. \\
\intertext{Replacing the index $i$ with $i+1$ in the third term, we obtain}
(RHS)&=  ky^{k}x^{\ell+1} 
   + \frac{k}{k+1} \sum_{i=1}^{k-1}\binom{\, k+1 \,}{i} y^{i}x^{\ell+1} 
   - \frac{1}{k+1} \sum_{i=1}^{k-1} \binom{\,k+1\,}{i+1}B_{k-i}
              \sum_{j=1}^{i} \binom{\, i+1 \,}{j} y^{j}x^{\ell+1}. \\
\intertext{Then, changing additive method in the third term, we obtain}
(RHS)&=  ky^{k}x^{\ell+1} 
   + \frac{k}{k+1} \sum_{i=1}^{k-1}\binom{\, k+1 \,}{i} y^{i}x^{\ell+1} 
   - \frac{1}{k+1} \sum_{j=1}^{k-1} \Biggl \{ 
        \sum_{i=j}^{k-1} \binom{\,k+1\,}{i+1} \binom{\, i+1 \,}{j} B_{k-i} \Biggr \} y^{j}x^{\ell+1}. \\
\intertext{In the third term, using Proposition \ref{prop} (i), we get}
(RHS)&=  ky^{k}x^{\ell+1} 
   + \frac{k}{k+1} \sum_{i=1}^{k-1}\binom{\, k+1 \,}{i} y^{i}x^{\ell+1} \\
&\quad - \frac{1}{k+1} \sum_{j=1}^{k-1} \Biggl \{ 
        \sum_{i=j}^{k-1} \binom{\,k+1\,}{j} \binom{\, k+1-j \,}{i+1-j} B_{k-i} \Biggr \} y^{j}x^{\ell+1}. \\
\intertext{Then, replacing the index $i+1-j$ with $i$ , we get}
(RHS)&=  ky^{k}x^{\ell+1} 
   + \frac{k}{k+1} \sum_{i=1}^{k-1}\binom{\, k+1 \,}{i} y^{i}x^{\ell+1} \\
&\quad - \frac{1}{k+1} \sum_{j=1}^{k-1} \Biggl \{ 
      \sum_{i=1}^{k-j} \binom{\,k+1\,}{j} \binom{\, k+1-j \,}{i} B_{k-(i+j-1)} \Biggr \} y^{j}x^{\ell+1} \\
&=  ky^{k}x^{\ell+1} 
   + \frac{k}{k+1} \sum_{i=1}^{k-1}\binom{\, k+1 \,}{i} y^{i}x^{\ell+1} \\
&\quad - \frac{1}{k+1} \sum_{j=1}^{k-1} \Biggl \{ 
      \sum_{i=1}^{k-j} \binom{\,k+1\,}{j} \binom{\, k-j+1 \,}{i} B_{k-j+1-i} \Biggr \} y^{j}x^{\ell+1}. \\
\intertext{Since we have $\binom{k-j+1}{i} = \binom{k-j+1}{k-j+1-i}$ , we can replace $B_{k-j+1-i}$ with $B_{i}$. Hence we have}
(RHS) &=  ky^{k}x^{\ell+1} 
   + \frac{k}{k+1} \sum_{i=1}^{k-1}\binom{\, k+1 \,}{i} y^{i}x^{\ell+1} \\
&\quad - \frac{1}{k+1} \sum_{j=1}^{k-1}  \binom{\,k+1\,}{j} \Biggl \{ 
      \sum_{i=1}^{k-j} \binom{\, k-j+1 \,}{i} B_{i} \Biggr \} y^{j}x^{\ell+1} \\
&=  ky^{k}x^{\ell+1} 
   + \frac{k}{k+1} \sum_{i=1}^{k-1}\binom{\, k+1 \,}{i} y^{i}x^{\ell+1} \\
&\quad - \frac{1}{k+1} \sum_{j=1}^{k-1}  \binom{\,k+1\,}{j} \Biggl \{ 
      \sum_{i=0}^{k-j} \binom{\, k-j+1 \,}{i} B_{i} - \binom{\, k+j-1 \,}{0}B_{0} \Biggr \} y^{j}x^{\ell+1}. \\
\intertext{In the third term, using Definition \ref{BN}, we get}
(RHS)&=  ky^{k}x^{\ell+1} 
   + \frac{k}{k+1} \sum_{i=1}^{k-1}\binom{\, k+1 \,}{i} y^{i}x^{\ell+1} 
   - \frac{1}{k+1} \sum_{j=1}^{k-1}  \binom{\,k+1\,}{j} \{ (k-j+1) - 1 \} y^{j}x^{\ell+1}. \\
\intertext{Then, replacing the index $j$ with $i$ in the third term, we have}
(RHS)&=  ky^{k}x^{\ell+1} 
   + \frac{k}{k+1} \sum_{i=1}^{k-1}\binom{\, k+1 \,}{i} y^{i}x^{\ell+1} 
   - \frac{1}{k+1} \sum_{i=1}^{k-1}  \binom{\,k+1\,}{i}(k-i) y^{i}x^{\ell+1} \\
&=  ky^{k}x^{\ell+1} 
   + \frac{1}{k+1} \sum_{i=1}^{k-1} k \, \binom{\, k+1 \,}{i} y^{i}x^{\ell+1} 
   - \frac{1}{k+1} \sum_{i=1}^{k-1}  \binom{\,k+1\,}{i}(k-i) y^{i}x^{\ell+1} \\
&=  ky^{k}x^{\ell+1} 
   + \frac{1}{k+1} \sum_{i=1}^{k-1} i \, \binom{\, k+1 \,}{i} y^{i}x^{\ell+1} \\
&=  \binom{k}{\,k-1\,}y^{k}x^{\ell+1} 
   + \sum_{i=1}^{k-1}  \binom{ k }{\,i-1\,} y^{i}x^{\ell+1} \\
&=  \sum_{i=1}^{k}  \binom{ k }{\,i-1\,} y^{i}x^{\ell+1}. \\
\intertext{Replacing the index $i$ with $i+1$, we get}
&=  \sum_{i=0}^{k-1}  \binom{ k }{\,i\,} y^{i+1}x^{\ell+1}. \\
\intertext{Regarding $y^{i+1}x^{\ell+1}$ as $y(y^{i}x^{\ell+1})$, we have }
&= y \sum_{i=0}^{k-1}  \binom{ k }{\,i\,} y^{i}x^{\ell+1} \\
&= y \Biggl ( \sum_{i=0}^{k}  \binom{ k }{\,i\,} y^{i}x^{\ell+1} - y^{k}x^{\ell+1} \Biggr ) \\
&= y \Biggl ( \sum_{i=0}^{k}  \binom{ k }{\,i\,} y^{i}x - y^{i}x \Biggr )x^{\ell} \\
&= y( xy^{k} - y^{k}x)x^{\ell}  \\
&= y \W_{k, \ell} = (LHS) . \\
\end{align*}
Therefore, we obtain desired results.
\end{proof}
\vspace{1\baselineskip}

From the theorem, we can get some corollaries. 
As has been mentioned in the introduction, $A$ is isomorphic to the universal enveloping algebra of a two-dimensional non-abelian Lie algebra. 
Thus, using PBW theorem, we can put 
\begin{equation*}
W= \bigoplus_{m \geq 1, n \geq 0} Kx^{m}y^{n}. 
\end{equation*}
Here we put
\begin{equation*} 
W'= \Biggl \{ \sum_{k,\ell} c_{k,\ell} \W_{k,\ell}  \Bigg |
   \begin{matrix}
       &k \geq 1, \ell \geq 0, c_{k,\ell} \in K, \\
       &\text{$c_{k,\ell}=0$ for all but finitely many pairs $(k,\ell)$}
   \end{matrix}
 \Biggr \} .
\end{equation*}
Then, the following statements hold.

\begin{Cor}
Notation is as above. 
Then, $W'$ is a two-sided ideal of $A$. 
In particular, $W=W'$. 
\end{Cor}

\begin{proof}
From Theorem \ref{thm}, it is clear that $W'$ becomes a left ideal of $A$. 
Again using Theorem \ref{thm}, we can see 
\begin{equation*}
W'=A\W_{1,0}=Ax. 
\end{equation*}
Then, we have 
\begin{align*}
W'x &= (Ax)x \subseteq W' \\
\intertext{and}
W'y &= (Ax)y = A(xy) = A(yx+x) =A(y+1)x \subseteq W' .
\end{align*}
Hence, $W'$ is a two-sided ideal of $A$. 
Using Proposition \ref{prop}, we can obtain
\begin{equation*}
x^{m}y^{n}=x^{m-1}(xy^{n}) = x^{m-1}\Biggl (\sum_{i=0}^{n} \binom{n}{i}y^{i} \Biggr )x, 
\end{equation*}
which implies $W=Ax$ and $W=W'$. 
Therefore, we obtain the desired result.
\end{proof}
\vspace{1\baselineskip}

Next, we see that $ \{ \W_{k,\ell} \}_{k \geq 1, \ell \geq 0}$ is a basis of $W'$.

\begin{Cor}
Notation is as above. Then, $ \{ \W_{k,\ell} \}_{k \geq 1, \ell \geq 0}$ is a basis of $W$, 
that is, $W= \oplus_{k\geq1, \ell \geq 0} K\W_{k,\ell}$. 
\end{Cor}

\begin{proof}

We show $\{ \W_{k,\ell} \}_{k\geq1, \ell \geq0}$ to be linearly independent. 
We assume 
\begin{equation*}
\sum_{\ell=1}^n \sum_{k=1}^m c_{k,\ell} (xy^{k}-y^{k}x)x^{\ell}=0 \quad( m,n < \infty)
\end{equation*}
with $c_{k,\ell} \in K$.
Then, from Proposition \ref{prop} (ii), we obtain
\begin{align*}
(LHS) &=\sum_{\ell=1}^n \sum_{k=1}^m c_{k,\ell} (xy^{k}-y^{k}x)x^{\ell} \\
        &= \sum_{\ell=1}^n \sum_{k=1}^m c_{k,\ell} \sum_{i=0}^{k-1} \binom{\,k\,}{i} y^{i}x^{\ell+1}. \\
\intertext{Hence, we have}
(LHS)    &= \sum_{\ell=1}^n  c_{m,\ell} \sum_{i=0}^{m-1} \binom{\,m\,}{i} y^{i}x^{\ell+1}
        +\sum_{\ell=1}^n \sum_{k=1}^{m-1} c_{k,\ell} \sum_{i=0}^{k-1} \binom{\,k\,}{i} y^{i}x^{\ell+1} \\              
    &= \sum_{\ell=1}^n  c_{m,\ell} \binom{m}{\,m-1\,} y^{m-1}x^{\ell+1}
        + \sum_{\ell=1}^n  c_{m,\ell} \sum_{i=0}^{m-2} \binom{\,m\,}{i} y^{i}x^{\ell+1}
        +\sum_{\ell=1}^n \sum_{k=1}^{m-1} c_{k,\ell} \sum_{i=0}^{k-1} \binom{\,k\,}{i} y^{i}x^{\ell+1}. \\              
\end{align*}
Then, the term $y^{m-1}x^{\ell+1}$ is appeared in the first term only. 
Using PBW theorem, we can get $c_{m, \ell} =0 $ for all $\ell$. 
Hence, the second term is vanished. That is, we have 
\begin{equation*}
(LHS) = \sum_{\ell=1}^n \sum_{k=1}^{m-1} c_{k,\ell} \sum_{i=0}^{k-1} \binom{\,k\,}{i} y^{i}x^{\ell+1}.               
\end{equation*}

Continuing this operation, we get $c_{k,\ell} = 0$ for all $k$.
Namely, we get $c_{k,\ell} = 0$ for all $k$ and $\ell$.
Hence, $\{ \W_{k,\ell} \}_{k\geq1, \ell \geq 0}$ is a basis of $W$.  
\end{proof}
\vspace{1\baselineskip}

Next, we show a variation of the Bernoulli-type relations. 

\begin{Cor}\label{cor3}
Let $A$ be as above.
We take 
\begin{equation*}
\W_{k} = xy^{k}-y^{k}x \in A \quad (k \geq 1).
\end{equation*}
Then, the following relations hold.
\begin{align*}
(S\!BR1) \quad y\W_{k} &= \frac{k}{k+1}\W_{k+1} -
                          \frac{1}{k+1} \sum_{i=1}^k \binom{\,k+1\,}{i}B_{k+1-i}\W_{i} \\
(S\!BR2) \quad \W_{k}y &= \frac{k}{k+1}\W_{k+1} -
                          \frac{1}{k+1} \sum_{i=1}^k (-1)^{k+1-i} \binom{\,k+1\,}{i}B_{k+1-i}\W_{i} \\
\end{align*}
\end{Cor}

\begin{proof}
In Theorem \ref{thm}, if we take $\ell=0$, then (SBR1) holds. 

Next, we show (BR2) by computing from (RHS) to (LHS). 
Using Proposition \ref{prop} (iii), we can compute
\begin{align*}
(RHS) &=\frac{k}{k+1} \W_{k+1} 
 - \frac{1}{k+1} \sum_{i=1}^{k} (-1)^{k+1-i} \binom{\,k+1\,}{i}B_{k+1-i}\W_{i} \\
&= \frac{k}{k+1} (xy^{k+1}-y^{k+1}x) 
   - \frac{1}{k+1} \sum_{i=1}^{k} \binom{\,k+1\,}{i}B_{k+1-i}  (xy^{i}-y^{i}x). \\
&= \frac{k}{k+1} \Biggl ( xy^{k+1}- \sum_{i=0}^{k+1}(-1)^{k+1-i}\binom{\,k+1\,}{i}xy^{i} \Biggr )\\
&\quad - \frac{1}{k+1} \sum_{i=1}^{k} (-1)^{k+1-i} \binom{\,k+1\,}{i}B_{k+1-i} 
                \Biggr ( xy^{i}- \sum_{j=0}^{i}(-1)^{i-j}\binom{\,i\,}{j}xy^{j} \Biggl ) \\
&= \frac{k}{k+1} \sum_{i=0}^{k}(-1)^{k-i}\binom{\,k+1\,}{i}xy^{i} \\
 &\quad - \frac{1}{k+1} \sum_{i=1}^{k} (-1)^{k+1-i} \binom{\,k+1\,}{i}B_{k+1-i} 
                           \sum_{j=0}^{i-1}(-1)^{i-1-j}\binom{\,i\,}{j}xy^{j}.  \\
\intertext{We divide (RHS) into three terms such as $x$ and $y^{k}x$ and otherwise. Then we have}
(RHS) &= \frac{(-1)^{k-k}k}{k+1} \binom{\,k+1\,}{k}xy^k 
            +  \frac{k}{k+1} \sum_{i=1}^{k-1}(-1)^{k-i}\binom{\,k+1\,}{i}xy^{i}
            + \frac{(-1)^{k-0}}{k+1} \binom{\,k+1\,}{0}xy^{0} \\
          &\quad- \frac{1}{k+1} \sum_{i=2}^{k} (-1)^{k+1-i} \binom{\,k+1\,}{i}B_{k+1-i} 
                           \sum_{j=1}^{i-1}(-1)^{i-1-j}\binom{\,i\,}{j}xy^{j}  \\
          &\quad- \frac{1}{k+1} \sum_{i=1}^{k} (-1)^{k+1-i} \binom{\,k+1\,}{i}B_{k+1-i}
                                 (-1)^{i-1-0}\binom{\,i\,}{0}xy^{0} . \\
\intertext{Since we can replace $B_{k+1-i}$ with $B_{i}$ in the last term, we have}
(RHS) &= kxy^k 
            +  \frac{k}{k+1} \sum_{i=1}^{k-1}(-1)^{k-i}\binom{\,k+1\,}{i}xy^{i}
            + \frac{(-1)^{k}}{k+1}x \\
            &\quad- \frac{1}{k+1} \sum_{i=2}^{k} (-1)^{k+1-i} \binom{\,k+1\,}{i}B_{k+1-i} 
                           \sum_{j=1}^{i-1}(-1)^{i-1-j}\binom{\,i\,}{j}xy^{j}  
                - \frac{(-1)^{k}}{k+1} \sum_{i=1}^{k} \binom{\,k+1\,}{i}B_{i}x \\
          &= kxy^k 
            +  \frac{k}{k+1} \sum_{i=1}^{k-1}(-1)^{k-i}\binom{\,k+1\,}{i}xy^{i} \\
          &\quad - \frac{1}{k+1} \sum_{i=2}^{k} (-1)^{k+1-i} \binom{\,k+1\,}{i}B_{k+1-i} 
                           \sum_{j=1}^{i-1}(-1)^{i-1-j}\binom{\,i\,}{j}xy^{j}  \\
            &\quad+ \frac{(-1)^{k}}{k+1}x    
              - \frac{(-1)^{k}}{k+1} \sum_{i=0}^{k} \binom{\,k+1\,}{i}B_{i}x 
              - \frac{(-1)^{k}}{k+1} \binom{\,k+1\,}{0}B_{0}x .\\
\intertext{In the fifth term, using Definition \ref{BN}, we get}
(RHS)&= kxy^k 
            +  \frac{k}{k+1} \sum_{i=1}^{k-1}(-1)^{k-i}\binom{\,k+1\,}{i}xy^{i} \\
          & \quad  - \frac{1}{k+1} \sum_{i=2}^{k} (-1)^{k+1-i} \binom{\,k+1\,}{i}B_{k+1-i} 
                           \sum_{j=1}^{i-1}(-1)^{i-1-j}\binom{\,i\,}{j}xy^{j}  \\
            &\quad + \frac{(-1)^{k}}{k+1}x    
              - \frac{(-1)^{k}}{k+1} (k+1)x 
              - \frac{(-1)^{k}}{k+1}x \\ 
          &= kxy^k 
            +  \frac{k}{k+1} \sum_{i=1}^{k-1}(-1)^{k-i}\binom{\,k+1\,}{i}xy^{i} \\
            &\quad - \frac{1}{k+1} \sum_{i=2}^{k} (-1)^{k+1-i} \binom{\,k+1\,}{i}B_{k+1-i} 
                           \sum_{j=1}^{i-1}(-1)^{i-1-j}\binom{\,i\,}{j}xy^{j} . 
\intertext{Replacing the index $i$ with $i+1$ in the third term, we obtain}
(RHS) &= kxy^k 
            +  \frac{k}{k+1} \sum_{i=1}^{k-1}(-1)^{k-i}\binom{\,k+1\,}{i}xy^{i} \\
           &\quad - \frac{1}{k+1} \sum_{i=1}^{k-1} (-1)^{k-i} \binom{\,k+1\,}{i+1}B_{k-i} 
                           \sum_{j=1}^{i}(-1)^{i-j}\binom{\,i+1\,}{j}xy^{j} . 
\intertext{Then, changing additive method in the third term, we obtain}
(RHS) &= kxy^k 
            +  \frac{k}{k+1} \sum_{i=1}^{k-1}(-1)^{k-i}\binom{\,k+1\,}{i}xy^{i} \\
           &\quad - \frac{1}{k+1}\sum_{j=1}^{k-1} \Biggl \{ 
            \sum_{i=j}^{k-1} (-1)^{k-j} \binom{\,k+1\,}{i+1}\binom{\,i+1\,}{j}B_{k-i} \Biggr \} xy^{j} .
\intertext{In the third term, using Proposition \ref{prop} (i), we get}
(RHS) &= kxy^k 
            +  \frac{k}{k+1} \sum_{i=1}^{k-1}(-1)^{k-i}\binom{\,k+1\,}{i}xy^{i} \\
           &\quad - \frac{1}{k+1}\sum_{j=1}^{k-1} \Biggl \{ 
            \sum_{i=j}^{k-1} (-1)^{k-j} \binom{\,k+1\,}{j}\binom{\,k+1-j\,}{i+1-j}B_{k-i} \Biggr \} xy^{j} . 
\intertext{Then, replacing the index $i+1-j$ with $i$ , we get}
(RHS) &= kxy^k 
            +  \frac{k}{k+1} \sum_{i=1}^{k-1}(-1)^{k-i}\binom{\,k+1\,}{i}xy^{i} \\
           &\quad  - \frac{1}{k+1}\sum_{j=1}^{k-1} \Biggl \{ 
            \sum_{i=1}^{k-j} (-1)^{k-j} \binom{\,k+1\,}{j}\binom{\,k+1-j\,}{i}B_{k-(i+j-1)} \Biggr \} xy^{j} \\ 
          &= kxy^k 
            +  \frac{k}{k+1} \sum_{i=1}^{k-1}(-1)^{k-i}\binom{\,k+1\,}{i}xy^{i} \\
           &\quad - \frac{1}{k+1}\sum_{j=1}^{k-1} \Biggl \{ 
            \sum_{i=1}^{k-j} (-1)^{k-j} \binom{\,k+1\,}{j}\binom{\,k-j+1\,}{i}B_{k-j+1-i} \Biggr \} xy^{j}.  
\intertext{Since we have $\binom{k-j+1}{i} = \binom{k-j+1}{k-j+1-i}$ , we can replace $B_{k-j+1-i}$ with $B_{i}$. Hence we have}
(RHS) &= kxy^k 
           +  \frac{k}{k+1} \sum_{i=1}^{k-1}(-1)^{k-i}\binom{\,k+1\,}{i}xy^{i} \\
           &\quad - \frac{1}{k+1}\sum_{j=1}^{k-1} (-1)^{k-j} \binom{\,k+1\,}{j} \Biggl \{ 
                            \sum_{i=1}^{k-j} \binom{\,k-j+1\,}{i}B_{i} \Biggr \} xy^{j} \\  
           &= kxy^k 
           +  \frac{k}{k+1} \sum_{i=1}^{k-1}(-1)^{k-i}\binom{\,k+1\,}{i}xy^{i} \\
           &\quad - \frac{1}{k+1}\sum_{j=1}^{k-1} (-1)^{k-j} \binom{\,k+1\,}{j} \Biggl \{ 
             \sum_{i=0}^{k-j} \binom{\,k-j+1\,}{i}B_{i} - \binom{\,k-j+1\,}{0}B_{0} \Biggr \} xy^{j}.  
\intertext{In the third term, using Definition \ref{BN}, we get}
(RHS) &= kxy^k 
           +  \frac{k}{k+1} \sum_{i=1}^{k-1}(-1)^{k-i}\binom{\,k+1\,}{i}xy^{i}
            - \frac{1}{k+1}\sum_{j=1}^{k-1} (-1)^{k-j} \binom{\,k+1\,}{j} \{ (k-j+1) - 1 \} xy^{j}.  
\intertext{Then, replacing the index $j$ with $i$ in the third term, we have}
(RHS) &= kxy^k 
           +  \frac{k}{k+1} \sum_{i=1}^{k-1}(-1)^{k-i}\binom{\,k+1\,}{i}xy^{i}
            - \frac{1}{k+1}\sum_{i=1}^{k-1} (-1)^{k-i} \binom{\,k+1\,}{i}(k-i) xy^{j} \\  
           &= kxy^k 
           +  \frac{1}{k+1} \sum_{i=1}^{k-1}(-1)^{k-i} k \binom{\,k+1\,}{i}xy^{i}
            - \frac{1}{k+1}\sum_{i=1}^{k-1} (-1)^{k-i} (k-i) \binom{\,k+1\,}{i} xy^{j} \\  
           &= kxy^k 
           +  \frac{1}{k+1} \sum_{i=1}^{k-1}(-1)^{k-i} i \binom{\,k+1\,}{i}xy^{i} \\  
           &= \binom{k}{\,k-1\,}xy^k 
           + \sum_{i=1}^{k-1}(-1)^{k-i}  \binom{k}{\,i-1\,}xy^{i} \\  
           &=  \sum_{i=1}^{k}(-1)^{k-i}  \binom{k}{\,i-1\,}xy^{i} .  
\intertext{Replacing the index $i$ with $i+1$, we get}
(RHS) &=  \sum_{i=0}^{k-1}(-1)^{k+1-i}  \binom{k}{\,i\,}xy^{i+1} .  
\intertext{Regarding $y^{i+1}x^{\ell+1}$ as $y(y^{i}x^{\ell+1}$, we have }
(RHS) &= \Biggl ( \sum_{i=0}^{k-1}(-1)^{k+1-i}  \binom{k}{\,i\,}xy^{i} \Biggr ) y \\
          &= \Biggl ( xy^{k} - \sum_{i=0}^{k}(-1)^{k-i}  \binom{k}{\,i\,}xy^{i} \Biggr ) y \\    
          &=  ( xy^{k} - y^{k}x ) y \\    
          &= \W_{k} y = (LHS). \\    
\end{align*}
Therefore, we obtain desired results.
\end{proof}

We can easily see that $\W_{k}$ is $\W_{k,0}$ in Theorem \ref{thm}. 
We will investigate connections between Bernoulli-type relations and Lie algebras in the next section. 
Using Corollary \ref{cor3}, we will show the formulas with respect to Lie algebras. 
\vspace{1\baselineskip}

\section{A Connection between Bernoulli-type relations and Lie algebras}

In this section, we consider a connection between the Bernoulli-type relations and Lie algebras. 
In the introduction, we roughly reviewed the classification of three-dimensional Lie algebras. 
We let $\LA$ be a three-dimensional Lie algebra over a field $K$ of characteristic zero  and denote by $U(\LA)$ the universal enveloping algebra of $\LA$. 
Then we also explain that if $\LA$ is the type (c), we have a two-dimensional Lie subalgebra $L$ of $\LA$. 
Then, $L$ is a non-abelian two-dimensional Lie algebra. 
That is, we can write $L=Kx \oplus Ky$ with $[x,y]=x$. 

Now, we recall our settings in Section 3. 
We let $K[x,y]$ be a noncommutative polynomial ring generated by $x,y$ 
and define $I = \langle xy-yx-x \rangle$ to be the ideal of $K[x,y]$ generated by $xy-yx-x$. 
We let $A$ be $K[x,y]/I$. 
Then, if we denote by $U(L)$ the universal enveloping algebra of $L$, then we can see the following :
\begin{Rem}\label{cong}
Notation is as above. Then we have 
$ A \cong U(L)$. 
\hfill$\Box$
\end{Rem}

From Remark \ref{cong}, we can use the Bernoulli-type relations for $U(L)$. 
Conversely, it is the reason that we can use PBW theorem in $A$. 
Using the relations in Section 3, we will show the next formulas in $U(L)$. 

\begin{Prop}
Let $L$ be as above. Then in $U(L)$, we have
\begin{align*}
\Pk \quad yxy^k &= \frac{k}{k+1} \, xy^{k+1} + \frac{1}{k+1} \, x^{k+1}x \\
  & \quad -  \frac{1}{k+1} \sum_{i=1}^{k} \, \binom{\,k+1 \,}{i} B_{k+1-i} \, xy^{i}
           +  \frac{1}{k+1} \sum_{i=1}^{k}  \, \binom{\,k+1 \,}{i} B_{k+1-i} \, y^{i}x, \\
\Qk \quad y^kxy &= \frac{1}{k+1} \, xy^{k+1} + \frac{k}{k+1} \, y^{k+1}x \\
 & \quad + \frac{1}{k+1} \sum_{i=1}^{k} (-1)^{k+1-i} \, \binom{\, k+1 \,}{i} B_{k+1-i} \, xy^{i} \\
 & \quad - \frac{1}{k+1} \sum_{i=1}^{k} (-1)^{k+1-i} \, \binom{\, k+1 \,}{i} B_{k+1-i} \, y^{i}x.
\end{align*}
\end{Prop}
\begin{proof}
Using the Corollary \ref{cor3}, we see that (SBR1) implies $\Pk$ and (SBR2) implies $\Qk$.
\end{proof}

\begin{Rem}
The above formulas $\Pk$ and $\Qk$ completely give the remaining terms of $\Ak$ and $\Bk$ in case of the type (c) if we replace $x$, $y$ by $e+g$, $f+g$ respectively. 
\end{Rem}

\begin{Rem}
Using the theory of linear algebras, 
we can establish two-dimensional Lie algebras as follows: 

Let $V$ be a vector space over $K$, and $End(V)$ be its endmorphism ring. 
Put $\Lg = End(V) \oplus V$, and we define 
\begin{equation*}
[f_1+\V_1, f_2+\V_2]=(f_1f_2-f_2f_1) + (f_1(\V_2)- f_2(\V_1))
\end{equation*}
for all $f_1, f_2 \in End(V)$ and $\V_1, \V_2 \in V$. 
Then $\Lg$ becomes a Lie algebra. 
Suppose that $f \in End(V)$ and $\V \in V$ satisfy $f(\V)=c\V$ for some $c \in K$. 
Put $\La = Kf \oplus K\V$ as a Lie subalgebra of $\Lg$. Then, we have 
\begin{equation*}
\begin{cases}
&\La \, \text{ is abelian}\quad \; \, ( \text{if }\; c =0 ), \\
&\La \cong L \quad \quad \quad \quad ( \text{if } \; c \neq 0 ). 
\end{cases}
\end{equation*}
\end{Rem}

\begin{Rem}
If $K$ is algebraically closed, then three-dimensional Lie algebras of type (d) corresponding to 
$\bigl ( \begin{smallmatrix} 1& \beta \\ 0&1 \end{smallmatrix}  \bigr )$ 
in Jacobson's book \cite{Ja}, on page 12, are not according to $\beta$. 
Hence, in this paper, we introduce the exact one type as (d)-($+$) at the introduction. 
\end{Rem}

\vspace{1\baselineskip}

\textbf{Acknowledgment  }

The author wish to express his gratitude to Professor J. Morita for his encouragement  and valuable advice.

\vspace{1\baselineskip}

\end{document}